\documentclass[11pt,reqno]{amsart}

\usepackage[text={155mm,245mm},centering]{geometry}   
\usepackage[utf8]{inputenc}
\usepackage[english]{babel}
\usepackage{amsmath}
\usepackage{amsfonts}
\usepackage{amssymb}
\usepackage{epstopdf}
\usepackage[all]{xy}
\usepackage{amsmath,amsthm}
\usepackage{amsmath,amssymb}

\usepackage{latexsym}

\usepackage{amsmath,amsthm}

\usepackage[mathscr]{eucal}
\usepackage[all]{xy}
\usepackage{mathrsfs}
\usepackage{hyperref}

\usepackage{color} 
\usepackage{soul}
\setulcolor{red}
\sethlcolor{yellow}
\setstcolor{blue}
\usepackage{amscd}
\usepackage{amsmath, amssymb}
\usepackage{amsfonts}

\allowdisplaybreaks

\newtheorem{thm}{Theorem}[section]
\newtheorem{cor}[thm]{Corollary}
\newtheorem{lem}[thm]{Lemma}
\newtheorem{prop}[thm]{Proposition}
\theoremstyle{definition}
\newtheorem{defn}[thm]{Definition}
\newtheorem{rem}[thm]{Remark}

\numberwithin{equation}{section}

\begin{document}

\title{On the moduli spaces of parabolic symplectic/orthogonal bundles on curves}

\author{Jianping Wang}
\address{{Address of author: School of Mathematical Sciences, University of Science and Technology of China, Hefei, 230026, China.}}
\email{jianpw@ustc.edu.cn}

\author{Xueqing Wen}

\address{{Address of author: Yau Mathematical Sciences Center, Beijing, 100084, China.}}
	\email{\href{mailto:email address}{{xueqingwen@mail.tsinghua.edu.cn}}}

\begin{abstract}

We prove that the moduli spaces of parabolic symplectic/orthogonal bundles on a smooth curve are globally F regular type. As a consequence, all higher cohomology of theta line bundle vanish. During the proof, we develop a method to estimate codimension, and consider the infinite grassmannians for parabolic $G$ bundles.

\end{abstract}

\maketitle

\section{Introduction}

Let $X$ be a variety over an algebraically closed field of positive characteristic and $F_X: X\rightarrow X$ be the absolute Frobenius map. In \cite{MR85} Mehta and Ramanathan introduced the notion ``\emph{F split}": $X$ is said to be F split if the natural morphism $F_X^{\#}: \mathcal{O}_X\rightarrow F_{X*}\mathcal{O}_X$ splits as an $\mathcal{O}_X$ module morphism. Later in \cite{S00} Smith studied a special kind of F split varieties: \emph{globally F regular} varieties(see Section \ref{section 6} for details). F split varieties and globally F regular varieties have many nice properties, for example, the vanishing of higher cohomologies of ample line bundles(nef line bundles in the case of globally F regular varieties).

Although almost all varieties are not F split, some important kind of varieties are, such as flag varieties, toric varieties. In \cite{MR96} Mehta and Ramadas proved that the moduli space of semistable parabolic rank two vector bundles with fixed determinant on a \emph{generic} nonsingular projective curve, are F split. They conjectured that the generic condition can be moved. On the other hand, as mentioned in \cite{SZh18}, this conjecture should be extended into the following: the moduli spaces of semistable parabolic bundles with fixed determinant on \emph{any} nonsingular projective curve are globally F regular.

In \cite{SZh18} Sun and Zhou studied the characteristic zero analogy of this extended conjecture. A variety over a field of characteristic zero is said to be of \emph{globally F regular type} if its modulo $p$ reduction are globally F regular for all $p\gg 0$. They proved that the moduli spaces of semistable parabolic vector bundles on a smooth projective curve over an algebraically closed field of characteristic zero are of globally F regular type. As an application, they can give a \emph{finite dimensional proof} of the so called Verlinde formula in $GL_n$ and $SL_n$ case(\cite{SZh20}).

Globally F regular type varieties have similar vanishing properties as globally F regular varieties, namely all higher cohomologies of nef line bundles are vanishing. Unlike the positive characteristic case, in characteristic zero, all Fano varieties with rational singularities are globally F regular type varieties(\cite{S00}). So globally F regular type varieties can be regarded as a generalization of Fano varieties in characteristic zero, with the vanishing properties retained and hence it is interesting to find examples of globally F regular type varieties.

On the other hand, properties of moduli spaces is a central topic in the study of moduli problems. We already know that for connected simply connected algebraic group $G$, the moduli space of semistable $G$ bundles on a smooth curve is a Fano variety (\cite{KN97}). However, if one consider the moduli space of semistable $G$ bundles with parabolic structure on a smooth curve, then one may not get a Fano variety. As mentioned before, in the case of $G=SL_n$, Sun and Zhou proved that the moduli spaces of semistable parabolic vector bundles with fixed determinant are globally F regular type varieties(\cite{SZh18}). So it encourage us to consider globally F regularity as a reasonable property of moduli spaces of $G$ bundles with parabolic structure on curves.

In this paper, we consider parabolic symplectic and orthogonal bundles over smooth curves. Our main theorem is the following:

\begin{thm}[Main theorem, see Theorem \ref{Main thm}]

The moduli spaces of semistable parabolic symplectic/orthogonal bundles over any smooth projective curve are globally F regular type varieties. As a consequence, any higher cohomologies of nef line bundles on these moduli spaces vanish.

\end{thm}

We now describe how this paper is organized:

In Section \ref{basics}, we recall some basics about parabolic vector bundles, parabolic symplectic/orthogonal bundles and the equivalence between parabolic bundles and orbifold bundles.

In Section \ref{construction}, we construct the moduli space of semistable parabolic symplectic/orthogonal bundles explicitly, using Geometric Invariant Theory.

In Section \ref{cod est}, we develop a technique to estimate the codimension of unsemistable locus in a given family, not only for parabolic symplectic/orthogonal bundles, but also $G$ bundles and parabolic vector bundles.

In Section \ref{Inf Gr}, to evaluate the canonical line bundle on the moduli spaces we constructed in Section \ref{construction}, we introduce the infinite Grassmannians for parabolic $G$ bundles, here $G$ is a connected simply connected simple algebraic group; we also define the theta line bundles for any family of symplectic/orthogonal bundle then we can show that under certain choice of rank and weights, the moduli spaces we constructed in Section \ref{construction} are Fano varieties.

In Section \ref{section 6}, we recall definition and properties of globally F regular type varieties, with the help of key Proposition \ref{key prop}, we can prove our main theorem.

\quad

\noindent\textbf{Acknowledgements} We would like to thank our supervisor, Prof. Xiaotao Sun, who brought this problem to us and kindly answer our questions. The second author would like to thank Dr. Bin Wang and Dr. Xiaoyu Su, for helpful discussions.

\section{Basics of parabolic principal bundle over curve}\label{basics}

\subsection{Parabolic vector bundles and parabolic symplectic/orthogonal bundles}

Let $C$ be a smooth projective curve of genus $g\geq 0$. We fix a reduced effective divisor $D$ of $C$, and an integer $K>0$.

$E$ is a vector bundle of rank $r$ and degree $d$ over $C$, by a parabolic structure on $E$, we mean the following:

\begin{enumerate}

\item[(1)] At each $x\in D$, we have a choice of flag of $E_x$:$$0=F_{l_x}(E_x)\subseteq F_{l_{x}-1}(E_x)\subseteq \cdots \subseteq F_0(E_x)=E_x$$ Let $n_i(x)=\text{dim}F_{i-1}(E_x)/F_i(E_x)$ and $\overrightarrow{n}(x)=\big(n_1(x),n_2(x),\cdots,n_{l_x}(x)\big)$. Notice that all these filtrations together are equivalent to a filtration:$$E(-D)=F_l(E)\subseteq F_{l-1}(E)\subseteq \cdots \subseteq F_0(E)=E$$

\item[(2)] At each $x\in D$, we fix a choice of sequence of integers, which are called weights: $$0\leq a_1(x)<a_2(x)\cdots <a_{l_x}(x) <K$$ Put $\overrightarrow{a}(x)=\big(a_1(x),a_2(x),\cdots, a_{l_x}(x)\big)$.

\end{enumerate}

We say that $\big(E, D, K, \{\overrightarrow{n}(x)\}_{x\in D}, \{\overrightarrow{a}(x)\}_{x\in D}\big)$, or simply $E$, is a parabolic vector bundle, and $\sigma=\big(\{\overrightarrow{n}(x)\}_{x\in D}, \{\overrightarrow{a}(x)\}_{x\in D}\big)$ is the parabolic type of $E$.

For any subbundle $F$ of the vector bundle $E$, it is clearly that there is an induced parabolic structure on $F$, with induced flags structures and same weights; similarly there is an induced parabolic structure on $E/F$.

Let $E_1$ and $E_2$ be two parabolic vector bundle with same weights, the space of parabolic homomorphisms $\text{Hom}_{par}(E_1,E_2)$ given by $\mathcal{O}_C$-homomorphisms between $E_1$ and $E_2$ preserving filtrations at each $x\in D$. We can also define parabolic sheaf of parabolic homomorphisms $\mathcal{H}om_{par}(E_1,E_2)$ in a similar way, which inherits a parabolic structure naturally. In fact, in \cite{Yo95}, it is shown that the category of parabolic bundles is contained in an abelian category with enough injectives. So we have the derived functors of parabolic homomorphism. We use $\text{Ext}^1_{par}(E_1,E_2)$ to denote the space of parabolic extensions.

\begin{defn}

The parabolic degree of $E$ is defined by $$pardegE=degE+\dfrac{1}{K}\sum_{x\in D}\sum_{i=1}^{l_x}a_i(x)n_i(x)$$
and $E$ is said to be stable(resp. semistable) if for all nontrivial subbundle $F\subset E$, concerning the induced parabolic structure, we have: $$\dfrac{pardegF}{rankF}<\dfrac{pardegE}{rankE}\ \ \ \ \ \ (\text{resp.}\ \leq)$$

\end{defn}

Now let us talk about family of parabolic vector bundles. Let $S$ be a scheme of finite type, a family of parabolic vector bundle with type $\sigma$ over $C$ parametrized by $S$ is a vector bundle $\mathcal{E}$ over $S\times C$, together with filtrations of vector bundles on $\mathcal{E}_x$ of type $\overrightarrow{n}(x)$ and weights $\overrightarrow{a}(x)$ for each $x\in D$. As before, such filtrations are equivalent to the following: $$\mathcal{E}\big(-(S\times D)\big)=F_l(\mathcal{E})\subseteq F_{l-1}(\mathcal{E})\subseteq \cdots \subseteq F_0(\mathcal{E})=\mathcal{E}$$ where $S\times D$ is considered as an effective divisor of $S\times C$. Following \cite{Yo95}, we say $\mathcal{E}$ is a flat family if all $F_i(\mathcal{E})$ are flat families.

\begin{defn}
$E$ is a vector bundle of rank $r$ degree $d$ over $C$. By a symplectic/orthogonal parabolic structure on $E$, we mean the following:

\begin{enumerate}

\item[(1)] A non-degenerated anti-symmetric/symmetric two form $$\omega: E\otimes E\longrightarrow \mathcal{O}_C(-D)$$

\item[(2)] At each $x\in D$, a choice of flag: $$0=F_{2l_x+1}(E_x)\subseteq F_{2l_x}(E_x)\subseteq \cdots F_{l_x+1}(E_x)\subseteq F_{l_x}(E_x)\subseteq \cdots \subseteq F_0(E_x)=E_x$$ where $F_i(E_x)$ are isotropic subspaces of $E_x$ respect to the form $\omega$ and $F_{2l_x+1-i}(E_x)=F_i(E_x)^{\perp}$ for $l_x+1\leq i \leq 2l_x+1$.

\item[(3)] At each $x\in D$, we fix a choice of weights: $$0\leq a_1(x)<a_2(x)\cdots <a_{l_x}(x)<a_{l_x+1}(x) <\cdots <a_{2l_x+1}(x) \leq K$$
satisfying $a_{i}(x)+a_{2l_x+2-i}(x)=K$, $1\leq i \leq l_x+1$.

\end{enumerate}

As before, we put $n_i(x)=\text{dim}\big(F_{i-1}(E_x)/F_i(E_x)\big)$ , and
\begin{align*}
\overrightarrow{n}(x)&=\big(n_1(x),n_2(x),\cdots,n_{2l_x+1}(x)\big)\\
\overrightarrow{a}(x)&=\big(a_1(x),a_2(x),\cdots,a_{2l_x+1}(x)\big)
\end{align*}
We say that $\big(E, \omega, D, K, \{\overrightarrow{n}(x)\}_{x\in D}, \{\overrightarrow{a}(x)\}_{x\in D}\big)$, or simply $E$, is a parabolic symplectic/orthogonal bundle and $\sigma=\big(\{\overrightarrow{n}(x)\}_{x\in D}, \{\overrightarrow{a}(x)\}_{x\in D}\big)$ is the parabolic type of $E$.
\end{defn}

\textbf{Convention:} when talked about parabolic symplectic/orthogonal bundles, we always assume that deg$D$ is even, and we fix a line bundle $L$ over $C$ and an isomorphism $L^{\otimes 2}\cong \mathcal{O}_C(D)$.

\begin{rem}
\
\begin{enumerate}

\item The original definition of parabolic principal bundles is just a principal bundle together with additional structures \cite{R96}. Later in \cite{BBN01} Balaji, Biswas and Nagaraj establish a different definition, which share some nice properties as in the case of parabolic vector bundles, for example, a parabolic symplectic/orthogonal bundle admits an Einstein–Hermitian
connection if and only if it is polystable(\cite{BML11}).

\item Although in our definition, $E$ is not a principal symplectic/orthogonal bundle, but $E\otimes L$ is, we call $E$ twisted orthogonal/symplectic bundle.

\end{enumerate}

\end{rem}

The parabolic degree of $E$ is given by
\begin{align*}
pardegE=degE+\dfrac{1}{K}\sum_{x\in D}\sum_{i=1}^{2l_x+1}a_i(x)n_i(x)
\end{align*}
By relations between $\overrightarrow{n}(x)$ and $\overrightarrow{a}(x)$ we see that $pardegE=degE+\frac{r}{2}degD$, noticing that $\omega: E\otimes E\rightarrow \mathcal{O}_X(-D)$ is non-degenerated, so $E\simeq E^{\vee}(D)$. Thus $degE+ \dfrac{r}{2}degD=0$ and then $pardegE=0$.

For any subbundle $F$ of $E$, we can define the parabolic degree of $F$ by $$pardegF=degF+\dfrac{1}{K}\sum_{x\in D}\sum_{i=1}^{2l_x+1}a_i(x)n_i^F(x)$$ where $n_i^F(x)=\text{dim}\big(F_{i-1}(E_x)\cap F_x/F_i(E_x)\cap F_x\big)$ .

\begin{defn}

A parabolic orthogonal/symplectic bundle $E$ is said to be stable(resp. semistable) if for all nontrivial isotropic subbundle $F\subset E$(by isotropic we mean $\omega(F\otimes F)=0$), we have $$pardegF<0 \ \ \ \ (\text{resp.}\ \leq)$$

\end{defn}

\begin{lem}\label{isotropic}

A parabolic symplectic/orthogonal bundle is semistable iff for any subbundle $F$, not necessarily isotropic, we have $pardegF\leq 0$, i.e. semistable as a parabolic vector bundle.

\end{lem}

\begin{proof}
 If E is semistable as a parabolic vector bundle, then it is semistable as  parabolic symplectic/orthogonal bundle.

Conversely, if E is a semistable parabolic symplectic/orthogonal bundle and a subbundle F is given. We want to show that $pardeg(F)\leq 0 $.

If $F \cap F^{\perp}=0$, then $E=F\oplus F^{\perp} $. Hence $2pardeg(F)=pardeg(F)+pardeg(F^{\perp})=deg(E)=0$ and we are done.

If $F \cap F^{\perp}\neq 0$, then we have the exact sequece of parabolic bundles:
\[0 \rightarrow F \cap F^{\perp} \rightarrow F\oplus F^{\perp} \rightarrow F+F^{\perp} \rightarrow 0\]
This shows that \[pardeg(F \cap F^{\perp})+pardeg(F + F^{\perp})=pardeg(F)+pardeg(F^{\perp})=2pardeg(F)\]

 It is easy to see $pardeg(F \cap F^{\perp}) \geq pardeg(F + F^{\perp})$ and hence we have $2pardeg(F) \leq 2pardeg(F \cap F^{\perp}) \leq 0$, where $pardeg(F \cap F^{\perp}) \leq 0$ since $F \cap F^{\perp}$ is isotropic.
\end{proof}

\subsection{Equivalence between parabolic bundles and orbifold bundles}
There is an interesting and useful correspondence between parabolic bundles and orbifold bundles, which is developed in \cite{MS80}, and \cite{B97} for general case. We will recall the correspondence briefly as follows:

Given $C, D, K$ as before, By Kawamata covering, there is a smooth projective curve $Y$ and a morphism $p: Y\rightarrow C$ such that $p$ is only ramified over $D$ with $p^*D=K\sum_{x\in D}p^{-1}(x)$, moreover, if we put $\Gamma=\text{Gal}\big(Rat(Y)/Rat(C)\big)$ to be the Galois group, then $p$ is exactly the quotient map of $Y$ by $\Gamma$.

\begin{defn}

An orbifold bundle over $Y$ is a vector bundle $W$ over $Y$ such that the action of $\Gamma$ lifts to $W$.

And an orbifold symplectic/orthogonal bundle is an orbifold bundle such that the correspondence 2-form $\omega$ is a morphism of orbifold bundles.

\end{defn}

Given an orbifold bundle $W$, for any $y=p^{-1}(x)\in p^*D$, the stabilizer $\Gamma_y$, which is a cyclic group of order $K$, acts on the fiber $W_y$ by some representation(after choosing suitable basis):$$\xi_K\longmapsto \text{diag}\{\xi_K^{a_1(x)},\cdots,\xi_K^{a_1(x)},\xi_K^{a_2(x)},\cdots,\xi_K^{a_{l_x}(x)}\}$$ where $0\leq a_1(x)< a_2(x)\cdots a_{l_x}(x)< K$ are integers, $\xi_K$ is the $K$-th root of unity and the multiplicity of $\xi_K^{a_i(x)}$ is given by $n_i(x)$. Similar in the definition of parabolic bundle, we use $\sigma=\big(\{\overrightarrow{n}(x)\}_{x\in D}, \{\overrightarrow{a}(x)\}_{x\in D}\big)$ to denote the type of orbifold bundle $W$.

\begin{prop}[\cite{MS80},\cite{B97}]\label{equivalence between pa and or}

There is an equivalence between the category of orbifold bundles over $Y$ with type $\sigma$ and the category of parabolic vector bundles over $C$ with type $\sigma$.

\end{prop}

Roughly speaking, given an orbifold bundle $W$, then $(p_*W)^{\Gamma}$ is a parabolic vector bundle over $C$, with parabolic structures given by the action of stablizers. Conversely, $E$ is a parabolic vector bundle, we put $W_1=p^*E$, after some elementary transformations of $W_1$, we would have an orbifold bundle of type $\sigma$. Moreover, we have $$\#\Gamma\cdot pardegE=degW$$ and E is (semi)stable as parabolic bundle if and only if W is (semi)stable as orbifold bundle.

 Now we will talk about orbifold symplectic/orthogonal bundles over $Y$: an orbifold symplectic/orthogonal bundles is symplectic/orthogonal bundle $W$ over $Y$ such that the action of  $\Gamma$ lifts to $W$ compatible with the symplectic/orthogonal structure. For any $y=p^{-1}(x)\in p*(D)$, the action of stabilizer is given by: $$\xi_K\longmapsto \text{diag}\{\xi_K^{a_1(x)},\cdots,\xi_K^{a_1(x)},\xi_K^{a_2(x)},\cdots,\xi_K^{a_{l_x}(x)},\xi_K^{-a_{l_x}(x)},\cdots, \xi_K^{-a_1(x)} \}$$ As before, we use $\sigma$ to denote the type of this orbifold symplectic bundle. Similarly,  we have:

\begin{prop}\label{symplectic orifold corr}

There is an equivalence between the category of orbifold symplectic/orthogonal bundles over $Y$ with type $\sigma$ and the category of parabolic symplectic/orthogonal bundles over $C$ with type $\sigma$. Moreover, this equivalence induces an equivalence between orbifold isotropic subbundles  and isotropic subbundles.

\end{prop}

\begin{proof}

See \cite{BML11}.
\end{proof}

\section{Moduli space of semistable parabolic symplectic/orthogonal bundles}\label{construction}

In this section, we construct the moduli space of semistable parabolic symplectic/orthogonal bundles with fixed parabolic type $\sigma$ over $C$. Although the moduli space is already constructed in \cite{BR89} for general algebraic groups, but for our purpose, we will construct the moduli spaces explicitly using GIT constructions.

\subsection{Construction of the moduli space}
In this section we will use $E$ to denote a parabolic symplectic/orthogonal bundle of rank $r$, degree $d$ and parabolic type $\sigma$.

We will fix an ample line bundle $\mathcal O(1)$ on $C$ with degree $c$, then the Hilbert polynomial of E is $P_{E}(m)=crm+\chi(E).$

Firstly we notice that by Lemma 2.3 of \cite{TI03}, the class of semistable parabolic orthogonal/symplectic bundles with fixed rank, degree and parabolic type are bounded. So we may choose an integer $N_0$ large enough so that $E(N)$ is globally generated for all semistable parabolic bundle $E$ and all integers $N\geq N_0$ ; which means, we have a quotient
\[q:V\otimes  \mathcal O_X(-N) \twoheadrightarrow E\]
where V is the vector space $\mathbb{C}^{P(N)}$ and $P$ is the Hilbert polynomial of E.

Let $Q$  be the Quot scheme of quotients of $V\otimes  \mathcal O_X(-N)$ with Hilbert polynomial $P$.

The orthogonal/symplectic structure on $E$ will induce a morphism :
\[(V\otimes \mathcal{O}_C)\otimes (V\otimes \mathcal{O}_C) \longrightarrow E(N)\otimes E(N) \longrightarrow \mathcal O_C(2N-D)\]
which is equivalent to a bilinear map on $V$: $$\phi: V\otimes V\longrightarrow \text{H}^0(C, \mathcal{O}_C(2N-D))$$
here $\mathcal{O}_C(2N-D)=\mathcal{O}_C(2N)\otimes \mathcal{O}_C(-D)$ and we use $H$ to denote the space $\text{H}^0(C, \mathcal{O}_C(2N-D))$.

Now we can regard every semistable $E$ as a point in the space $Q\times \mathbb{P}Hom(V\otimes V,H).$ However, not every element in $\mathbb{P}Hom(V\otimes V,H)$ would give a nondegenerated form on $E$. To fix this, we will use the following lemma:

\begin{lem}[Lemma 3.1 of \cite{TI03}]

Let $X$ be a smooth projective variety and $Y$ be a scheme. Consider a morphism of sheaves $f: \mathcal{E}\rightarrow \mathcal{F}$ over $X\times Y$, moreover, we assume $\mathcal{F}$ is flat over $Y$. Then there is a unique closed subscheme $Z$ of $Y$ satisfying the following universal property: for any scheme $S$ and a Cartesian diagram:

\begin{equation*}
\xymatrix{
& X\times S \ar[r]^{\bar{h}} \ar[d]^{p_S} & X\times Y\ar[d]^{p} \\
& S \ar[r]^{h} & Y}
\end{equation*}
then $\bar{f}^*(f)=0$ if and only if $h$ factors through $Z$.

\end{lem}

Now we let $Z\subset Q\times \mathbb{P}Hom(V\otimes V,H)$ be the closed subscheme such that every closed point ($q:V\otimes  \mathcal O_X(-N) \twoheadrightarrow E,\phi:V\otimes V\rightarrow  H$) of $Z$ represents a twisted symplectic/orthogonal bundle $E$.

So over $Z\times C$, we have a universal quotient $q:V\otimes p_C^*\mathcal{O}_C(-N)\rightarrow \mathcal{E}\rightarrow 0$ on $X\times Z$ and a nondegenerated anti-symmetric/symmetric two form $\omega: \mathcal{E}\otimes \mathcal{E}\rightarrow p_C^*\mathcal{O}_C(-D)$ where $p_C: Z\times C\rightarrow C$ is the projection. For any $x\in D$, let $\mathcal E_x$ be the restriction of $\mathcal E$ on $Z\times \{x\}\cong Z$ and we put $Flag_{\vec n(x)}(\mathcal E_x)\rightarrow Z$ be the relative isotropic flag scheme of  type $\vec n(x)$.

Let $\mathcal R:= \underset{x\in D}{\times_Z}Flag_{\vec n(x)}(\mathcal E_x)\rightarrow Z,$ then a closed point of $\mathcal R$ is represented by $$((q,\phi),(q_1(x),q_2(x),\dots,q_{2l_x}(x))_{x\in D}) $$ where $(q,\phi)$ is a point of Z,  and $q_i(x)$ is the composition $q_i(x):V \otimes \mathcal O_X(-N) \rightarrow E \rightarrow E_x\twoheadrightarrow Q_i(x)$. We denote by $Q_i(x)$ the quotients ${E_x}/{F_i(E)_x}$, and let $r_i(x)=dimQ_i(x).$

For $m$ large enough, let $\mathcal{G}=Grass_{P(m)}(V\otimes W_m)\times \mathbb{P}Hom(V\otimes V,H)\times \textbf{Flag}$, where, $W_m= H^0(V\otimes  \mathcal{O}(m-N))$, and $\textbf{Flag}$ is defined as:
\[\textbf{Flag}=\prod_{x\in D}\big(Grass_{r_1(x)}(V) \times \dots \times Grass_{r_{2l_x}(x)}(V)\big)\]

Now, consider the $SL(V)-$equivariant embedding
\[\mathcal R \hookrightarrow \mathcal G=Grass_{P(m)}(V\otimes W_m)\times \mathbb P Hom(V\otimes V,H)\times \textbf{Flag}\]
Which maps the point
$((q,\phi),(q_1(x),q_2(x),\dots,q_{2l_x}(x))_{x\in D})$ of $\mathcal R$ to the point
$$(g,\phi,(g_1(x),g_2(x),\dots,g_{2l_x}(x))_{x\in D})$$ of $\mathcal G$, where $g: V\otimes W_m\twoheadrightarrow H^0(E(m-N))$ and $ g_i(x): V\twoheadrightarrow Q_i(x)$.

We give the polarisation on $\mathcal G$ by:
\[n_1 \times 1 \times \prod_{x\in D} \prod_{i=1}^{ 2l_x}d_i(x)\]
Where $n_1=\mathnormal{\frac{l+KcN}{c(m-N)}} ,  d_i(x)=a_{i+1}(x)-a_i(x)$ and $\mathnormal l$ is the number satisfying
\[ \mathnormal{\sum_{x\in D}\sum_{i=1}^{2l_x}d_i(x)r_i(x)+rl=K\chi}\]

We will analyse the action of $SL(V)$ on $\mathcal{R}$ using a method in \cite{TI03}. Let $\mathcal{R}^{s}$(resp. $\mathcal{R}^{ss}$) to denote the sublocus of $\mathcal{R}$ where the corresponding parabolic symplectic/orthogonal bundles are stable(resp. semistable) and the map $\text{H}^0(q):V\rightarrow \text{H}^0(C, E(m))$ is an isomorphism. We are going to show $\mathcal{R}^{s}$(respectively, $\mathcal{R}^{ss}$) is the stable (respectively, semistable) locus of the action in the sense of GIT.  Firstly let us recall a definition in \cite{TI03}: 

\begin{defn} A weighted filtration $(E_{\textbf{•}}, m_{\textbf{•}})$ of a parabolic symplectic/orthogonal bundle $E$ consists of

\begin{enumerate}
\item[(1)] a filtration of subsheaves
\[0\subset E_1 \subset E_2 \subset \dots \subset E_t \subset E_{t+1}=E \]
We denote $rk(E_i)$ by $s_i$;

\item[(2)] a sequence of positive numbers $m_1,m_2 \dots ,m_t$, called the weights of this filtraion.

\end{enumerate}
\end{defn}

Let $\Gamma=\sum_{i=1}^{t}m_i \Gamma^{s_i}\in \mathbb C^r$, where
\[\Gamma^k=( \overbrace{k-r,k-r,\dots ,k-r}^{k},\overbrace{k\dots ,k}^{r-k})\]

Now, given a weighted filtration $(E_{\textbf{•}}, m_{\textbf{•}})$ of a parabolic symplectic/orthogonal bundle $E$, let $\Gamma_{j}$ be the $j$-th component of $\Gamma$, and we define \[\mu(\omega,E_{\textbf{•}}, m_{\textbf{•}}):=min \{ \Gamma_{s_{i_1}}+\Gamma_{s_{i_2}} : \omega \vert_{E_{i_1} \otimes E_{i_2}} \neq 0\} \]

We have the following result(see \cite{TI03}, Lemma 5.6):

\begin{lem}\label{nondegenerate}

 If $\omega$ is nondegenerate, then $\mu(\omega,E_{\textbf{•}}, m_{\textbf{•}})\leq 0$.

\end{lem}

\begin{proof}
we can take the index i and j such that $\mu(\omega,E_{\textbf{•}}, m_{\textbf{•}})= \Gamma_{s_i}+\Gamma_{s_j}$ and $\omega \vert_{E_i \otimes E_j} \neq 0$. Then there exist a point $x\in C$ away from $D$ such that the restriction $\omega_x = \omega \vert_{E_{i,x} \otimes E_{j,x}} \neq 0$.

Let $W=E_x$. Then the nondegenerate form $\omega$ over $E$ induces a nondegenerate form over the vector space $W$. We still write this form as $\omega :W \otimes W \rightarrow \mathbb C.$  By the nondegenrate condition, using Hilbert-Mumford criterion, one can see that $\overline {\omega} \in \mathbb {P}Hom(W \otimes W, \mathbb C)$ is GIT semistable with the natural $SL(W)$ action. It implies that $\mu(\omega,W_{\textbf{•}}, m_{\textbf{•}})\leq 0$ for all weighted filtrations of W.

It's easy to see $\mu(\varphi,E_{\textbf{•}}, m_{\textbf{•}}) \leq \mu(\omega,W_{\textbf{•}}, m_{\textbf{•}}),$ hence $\mu(\omega,E_{\textbf{•}}, m_{\textbf{•}}) \leq 0 $.
\end{proof}

In the following we use Hilbert-Mumford criterion to determine the (semi)stable locus for the action of $SL(V)$ of $\mathcal{R}$.

\begin{prop}\label{Hilbert Mumford}

A point $((q,\phi),(q_1(x),q_2(x),\dots,q_{2l_x}(x))_{x\in D})$ of $\mathcal R$ is GIT stable (resp. GIT semistable) for the action of $SL(V)$, with respect to the polarisation defined in definition 2.1, if and only if for all weighted filtration $(E_{\textbf{•}}, m_{\textbf{•}})$, we have
\[kP(N)(\sum_{i=1}^{t}(pardeg(E_i))+\mu(\omega,E_{\textbf{•}}, m_{\textbf{•}})<0\ \ (\text{resp.} \leq)\]

\end{prop}

\begin{proof}

By the Hilbert-Mumford criterion, a point $((q,\phi),(q_1(x),q_2(x),\dots,q_{2l_x}(x))_{x\in D})$ is GIT semistable if and only if any one parameter subgroup $\lambda: \mathbb{G}_m \rightarrow SL(V)$, the corresponding Hilbert-Mumford weight is greater or equal than zero. But a one parameter subgroup of $SL(V)$ is equivalent to a weighted filtration of $V$ and hence gives a weight filtration $(E_{\textbf{•}}, m_{\textbf{•}})$ for the corresponding bundle $E$. In terms of weight filtration for $E$, we see that the Hilbert-Mumford weight is given by 

\[ s(E):= n_1(\sum_{i=1}^t m_i(\chi (E_i(N))P(m)-P(N) \chi(E_i(m))))+\mu (\omega, E_{\textbf{•}}, m_{\textbf{•}})\]

\[+\sum_{x\in D}\sum_{j=1}^{2l_x}d_j(x)(\sum_{i=1}^t m_i(\chi (E_i(N))r_j(x)-P(N)r_j^{E_i}(x))) \]
where $r_j^{E_i}(x):= dim(Im(E_i \rightarrow E \rightarrow Q_j(x) ))$. Hence the point is GIT semistable if and only if $s(E)\leq 0$.

However, one can show that (see Proposition 2.9 of \cite{Sun17})
\[s(E)=kP(N) \big(\sum_{i=1}^{t}m_i pardeg(E_i)\big)+\mu(\omega, E_{\textbf{•}}, m_{\textbf{•}} )\]
In fact, the coefficients of $m_i$ in $s(E)$ is
 $$n_1 \big (\chi (E_i(N))P(m)-P(N) \chi(E_i(m)) \big )+\sum_{x\in D}\sum_{j=1}^{2l_x}d_j(x)\big (\chi (E_i(N))r_j(x)-P(N)r_j^{E_i}(x)\big )$$
\begin{equation*} \begin{split}
&=(rl+rKcN)\big ( deg(E_i)-\dfrac {r(E_i)}{r}deg(E) \big )+\sum_{x\in D}\sum_{j=1}^{2l_x}d_j(x)\big (\chi (E_i(N))r_j(x)-P(N)r_j^{E_i}(x)\big ) \\
&=K P(N)\big ( deg(E_i)- \dfrac{ r(E_i)}{r}deg(E) \big )+ P(N)\big( \dfrac {r(E_i)}{r}\sum_{x\in D}\sum_{j=1}^{2l_x}d_j(x) r_j(x)-\sum_{x\in D}\sum_{j=1}^{2l_x}d_j(x) r_j^{E_i}(x) \big) \\
&=K P(N)\big ( deg(E_i)- \dfrac{ r(E_i)}{r}deg(E) \big )+\dfrac {r(E_i)}{r}P(N) \big (r \sum_{x \in D}a_{2l_x+1}-\sum_{x\in D}\sum_{j=1}^{2l_x+1}a_j(x)n_j(x) \big ) \\
&\quad -P(N)\big (r(E_i) \sum_{x \in D}a_{2l_x+1}-\sum_{x\in D}\sum_{j=1}^{2l_x+1}a_j(x)n_j^{E_i}(x) \big ) \\
&=KP(N)(pardeg(E_i))\\
\end{split} 
\end{equation*}
\end{proof}

\begin{prop}

A parabolic symplectic/orthogonal bundle $E$ is stable(resp. semistable) if and only if the correspondence point $((q,\phi),(q_1(x),q_2(x),\dots,q_{2l_x}(x))_{x\in D})$ of $\mathcal R$ is GIT stable(resp. semistable) for the action of $SL(V)$.

\end{prop}

\begin{proof}
Let $E$ be a stable(resp. semistable) bundle. For any weighted filtration $(E_{\textbf{•}}, m_{\textbf{•}})$, we have $pardeg(E_i)<0\ (\text{ resp.}\ \leq)$ by Lemma \ref{isotropic}. Furthermore,by Lemma \ref{nondegenerate}, $\mu(\omega, E_{\textbf{•}}, m_{\textbf{•}})\leq 0$,hence
\[kP(N)(\sum_{i=1}^{t}(pardeg(E_i))+\mu(\omega,E.,m.)<0\ (\text{resp.}\ \leq)\]
By Proposition \ref{Hilbert Mumford}, this tells that the corresponding point $((q,\phi),(q_1(x),q_2(x),\dots,q_{2l_x}(x))_{x\in D})$ of $\mathcal R$ is GIT stable(resp. semistable).

Conversely, Let $E$ be a parabolic orthogonal/symplectic bundle such that the corresponding point $((q,\phi),(q_1(x),q_2(x),\dots,q_{2l_x}(x))_{x\in D})$ is GIT stable(resp. GIT semistable). We want to show that E is a stable (resp. semistable). That is, for any isotropic subbundle F of E, we have $pardeg(F)<0\ (\text{resp.}\ \leq)$.

Since $E$ is stable(resp. semistable), the inequality in Proposition \ref{Hilbert Mumford} must hold for all weighted filtrations $(E_{\textbf{•}}, m_{\textbf{•}})$. In particular, if we take the weighted filtration as:
$0 \subset F \subset F^{\perp} \subset E $, and weights $m_1=m_2=1$, then the inequality becomes
\[kP(N)((pardeg(F)+pardeg(F^{\perp}))+\mu(\omega, E_{\textbf{•}}, m_{\textbf{•}})<0\ (\text{ resp.}\ \leq).\]
However,  in this case we have $\mu(\omega,E.,m.)=0$ and $pardeg(F)=pardeg(F^{\perp})$, hence we have $pardeg(F)<0\ (\text{resp.}\ \leq).$
\end{proof}

Therefore,let $\mathcal{R}^{ss}\subset \mathcal R$ be the open set of $\mathcal R$ which consists of semistable parabolic orthogonal(symplectic,resp) sheaves. In the rest of this section, we will show that $\mathcal{R}^{ss}$ is smooth. Therefore let $ M_{G,P}=\mathcal R^{ss}//SL(V)$ be the GIT quotient, then we have

\begin{thm}

 $M_{G,P}$ is the coarse moduli space of semistable parabolic orthogonal/symplectic sheaves of rank $r$ and degree $d$ with fixed parabolic type $\sigma$. Moreover, $M_{G,P}$ is a normal Cohen-Macaulay projective variety, with only rational singularities.
 
\end{thm}

\begin{proof}
Since we have show that $\mathcal{R}^{ss}$ is smooth in the next subsection, especially $\mathcal{R}^{ss}$  is normal with only rational singularities, so is its GIT quotient $M_{G,P}$. Finally the fact that $\mathcal{R}^{ss}$ is regular implies that $M_{G,P}$ is Cohen-Macaulay(see \cite{HR74}).
\end{proof}

\subsection{Smoothness of $\mathcal{R}^{ss}$ }
The smoothness of $\mathcal{R}^{ss}$ has essentially proved in \cite{R96}. We will reformulate the proof here.

Let $Q_F$ be the open subscheme of $Q$ consisting of quotients $[q:V\otimes  \mathcal O_X(-N) \twoheadrightarrow E]\in Q$ such that $H^1(E(N))=0$. Let $Z_F$ be the inverse image of $Q_F$ under the projection $Z \rightarrow Q$ and $R_F$ be the inverse image of $Z_F$ under the projection $R \rightarrow Z$. If we can show that $Z_F$ is smooth, then $\mathcal{R}_F$ is smooth because is a flag bundle over $Z_F$. Thus $\mathcal{R}^{ss}$ is smooth as it is an open subscheme of  $Z_F$. So the smoothness of $\mathcal{R}^{ss}$ reduce to the smoothness of $Z_F$. We will prove the smoothness of $Z_F$ in the rest part of this subsection. First of all, let us recall the definition of Atiyah bundle of a principal $G$-bundle (\cite{B10})

\begin{defn}
	Let $p: E \rightarrow X$ be a principal $G$-bundle, the Atiyah bundle $At(E)$ of $E$ is defined as $At(E)(U):=H^0(p^{-1}U, T_{p^{-1}U})^G$ for any open subset $U \subseteq X$.
\end{defn}

\begin{prop}\label{Atiyah}
	Let $p: E \rightarrow X$ be a principal $G$-bundle and $At(E)$ is the Atiyah bundle, then 
	\begin{enumerate}
		\item[(1)] We have an exact sequence (the Atiyah sequence): $0 \rightarrow ad(E) \rightarrow At(E) \rightarrow T_X \rightarrow 0$.
		\item[(2)] There is a natural isomorphism $\mu: p^{*}At(E)\stackrel{\cong} {\longrightarrow} T_E$.
	\end{enumerate}
\end{prop}

\begin{proof}
	See \cite{B10} section 1.
\end{proof}

\begin{rem}\label{grass}
	For the Grassmann variety $G_{n,r}$, let $p: \mathcal{A} \rightarrow G_{n,r} $ be the universal $GL(r)$ bundle and $0 \rightarrow \mathcal{K} \rightarrow V \otimes \mathcal{O}_{G_{n,r}} \rightarrow \mathcal{A} \rightarrow 0$ be the universal exact sequence, then we  have an isomorphism $At(\mathcal{A}) \cong V \otimes \mathcal{A}$ and the Atiyah sequence becomes $0 \rightarrow \mathcal{A}^{*} \otimes \mathcal{A} \rightarrow V \otimes \mathcal{A} \rightarrow \mathcal{K}^{*} \otimes \mathcal{A}  \rightarrow 0$.
\end{rem}

Let $G \hookrightarrow GL(r)$ be the orthogonal/symplectic subgroup of $GL(r)$. In \cite{R96}, the author has construct the moduli space of principal $G$-bundles. The author also shows that  $Q_F$ and $Z_F$ can be the open subschemes of some Hilbert scheme:

Consider the Grassmann variety $G_{n,r}$, where $n=dim V$. Denote the universal family over $G_{n,r}$ by $\mathcal{A} \rightarrow G_{n,r}$. Let $Y=GL(r)//G$ and $\mathcal{A}(Y)=\mathcal{A}//G$ be  the fibre bundle with fibre $Y$ over $G_{n,r}$ associated to $\mathcal{A} \rightarrow G_{n,r}$. Then we have the following proposition:

\begin{prop}\label{open}
	$Q_F$ is an open subscheme of $Hom(C,G_{n,r})$ and $Z_F$ is an open subscheme of $Hom\big(C,\mathcal{A}(Y)\big)$.
\end{prop}

\begin{proof}
	See \cite{R96} section 4.13.
\end{proof}

\begin{prop}
The semistable locus $\mathcal{R}^{ss}$ is smooth.
\end{prop}

\begin{proof}
	As mentioned before, we just need to show that $Z_F$ is smooth. By proposition \ref{open}, $Z_F \subset Hom\big(C,\mathcal{A}(Y)\big) $. So let $f:C \rightarrow \mathcal{A}(Y)$ be a point of $Z_F$, We need to show that $Z_F$ is smooth at $f$, that is $Hom\big(C,\mathcal{A}(Y)\big)$ is smooth at $f$. However by associating $f$ to the graph  $\Gamma_f$, we may  consider $Hom\big(C,\mathcal{A}(Y)\big)$ as an open subscheme of $Hilb \big(C \times \mathcal{A}(Y)\big)$. Hence we should prove the that $Hilb \big(C \times \mathcal{A}(Y)\big)$ is smooth at $\Gamma_f$. By obstruction theory, this is equivalent to show that $H^1(\Gamma_f, N_{\Gamma_f})=0$ where $N_{\Gamma_f}$ is the normal bundle of $\Gamma_f$ in $C \times \mathcal{A}(Y)$. However, since $\Gamma_f \simeq C$ by the projection $C \times \mathcal{A}(Y) \rightarrow C$, we have $N_{\Gamma_f} \simeq f^{*}T_{\mathcal{A}(Y)}$ where $T_{\mathcal{A}(Y)}$ is the tangent bundle of $\mathcal{A}(Y)$. We will show that $H^1(C,f^{*}T_{\mathcal{A}(Y)})=0$.

	Let $p:\mathcal{A}(Y) \rightarrow G_{n,r}$ and $q:\mathcal{A} \rightarrow G_{n,r}$  be the caonical maps. Dnote by $ \theta : \mathcal{A} \rightarrow \mathcal{A}(Y)=\mathcal{A}//G $  the natural quotient.  It is easy to see $q=p \circ \theta$. Then we get the following exact sequence by taking the differential of the projection  $ \theta : \mathcal{A} \rightarrow \mathcal{A}(Y)$ 
	\begin{align*}
	0 \longrightarrow \mathcal{M} \longrightarrow \theta_*T_{\mathcal{A}} \longrightarrow T_{\mathcal{A}(Y)} \longrightarrow 0
	\end{align*}
	
	On the other hand, by proposition \ref{Atiyah}, we have 
	\begin{align*}
	T_{\mathcal{A}} \cong q^*At(\mathcal A) = (p \circ \theta)^*At(\mathcal A)=\theta^* p^*At(\mathcal A)
	\end{align*} 
	
	So we have a surjective map
	\begin{align*}
	\theta_* \theta^* p^*At(\mathcal A) \longrightarrow T_{\mathcal{A}(Y)}
	\end{align*}
	
	Compose the above map with the canonical map $ p^*At(\mathcal A) \longrightarrow \theta_* \theta^* p^*At(\mathcal A) $, we get a surjective morphism
	\begin{align*}
	p^*At(\mathcal A) \longrightarrow T_{\mathcal{A}(Y)}
	\end{align*}
	
	Assume the kernel is $\mathcal N$, then we have an exact sequence
	\begin{align*}
	0 \longrightarrow \mathcal{N} \longrightarrow  p^*At(\mathcal A) \longrightarrow T_{\mathcal{A}(Y)} \longrightarrow 0
	\end{align*}
	
	Taking the pullback functor $f^*$, we get an exact sequence over $C$
	\begin{align*}
	0 \longrightarrow f^*\mathcal{N} \longrightarrow  f^*p^*At(\mathcal A) \longrightarrow f^*T_{\mathcal{A}(Y)} \longrightarrow 0
	\end{align*}
	
	We get a exact sequence of cohomologies from the above sequence
	\begin{align}
	H^1(C,f^*\mathcal{N}) \longrightarrow H^1(C,f^*p^*At(\mathcal A)) \longrightarrow H^1(C,f^*T_{\mathcal{A}(Y)}) \longrightarrow 0
	\end{align}
	
	However, by Remark \ref{grass}, we have $At(\mathcal A)  \cong V \otimes \mathcal{A}$. So $$ f^*p^*At(\mathcal A) \cong (p \circ f )^* (V \otimes \mathcal{A})$$
	
	Notice that $p\circ f \in Z_F$, so it correspondence to a quotient bundle
	\begin{align*}
	0 \longrightarrow F_f \longrightarrow V \otimes \mathcal{O}_C(-N) \longrightarrow E_f \longrightarrow 0
	\end{align*}
	
	Then we have $(p \circ f )^* (V \otimes \mathcal{A}) \cong V \otimes E_f(N)$, so we have that 
	\begin{align}
	H^1\big(C,f^*p^*At(\mathcal A)\big) \cong H^1\big(C,V \otimes E_f(N)\big)=0 
	\end{align}
	
	Combine with (3.1) and (3.2) we finally get $H^1(C,f^*T_{\mathcal{A}(Y)})=0$.
	
\end{proof}

\section{Codimention estimate}\label{cod est}

In this section, we fix $S$ to be a scheme of finite type. Let $\mathcal{E}$ be a flat family of vector bundle, principal $G$ bundle, parabolic vector bundle or parabolic symplectic/orthogonal bundle over $C$ parametrized by $S$, under certain conditions, we want to estimate the codimension of the unstable (unsemistable) locus, i.e. the locally closed subscheme $S^{us}\subset S$ ($S^{uss}\subset S$) parametrizing all $\mathcal{E}_t$ which is not stable (semistable). Our main method is taken from \cite{PV85}.

\subsection{The case of vector bundle and principal $G$ bundle}

In fact, the case of vector bundle and principal $G$ bundle have been already done in \cite{PV85} and \cite{KN97}. For later use, we reformulate the results and give a short proof if necessary.

All the stories begin with the following proposition:

\begin{prop}\label{begin prop}

$\mathcal{E}$ is a flat family of vector bundles over $S\times C$. Let $\phi: Q\rightarrow S$ be the relative Quot-scheme parametrizing all flat quotients of $\mathcal{E}$ with certain fixed rank and degree. For any $s\in S$ and $q\in \phi^{-1}(s)$, corresponding to exact sequence: $$0\longrightarrow F\longrightarrow \mathcal{E}_s\longrightarrow G\longrightarrow 0$$ we have the following exact sequence:
\begin{align}\label{exact sequence}
0\longrightarrow \text{Hom}(F,G)\longrightarrow T_qQ\longrightarrow T_sS\longrightarrow \text{Ext}^1(F,G).
\end{align}

\end{prop}

\begin{proof}

See \cite{HL10} Proposition 2.2.7.
\end{proof}

Let $E$ be a vector bundle over $C$, the classical Harder-Narasimhan filtration and Jordan-Holder filtration show that if $E$ is not stable(resp. semistable), then there is a maximal stable subbundle $F_0\subset E$ with the property $\text{deg}\mathcal{H}om(F_0,E/F_0)\leq 0$ (resp. $< 0$). $F_0$ is taken to be the first term of the Jordan-Holder filtration of  the maximal destabilizing subbundle of $E$(so different choice of $F_0$ have same slope). Moreover, if we say $F_0$ is of type $\mu=(r',d')$, i.e. $F$ is of rank $r'$ and degree $d'$, Then for a flat family of vector bundle $\mathcal{E}$ over $S\times C$, the locus $S^{\mu}\subset S$ parametrizing $\mathcal{E}_t$ having a subbundle described above with type $\mu$, is locally closed and non-empty for finitely many $\mu$.

Similarly properties hold for principal $G$ bundles. Let $E$ be a principal $G$ bundle, then there is a unique standard parabolic subgroup $P$ and a unique reduction $E_P$, and if we denote $E_{\mathsf{s}}$ to be the vector bundle associated to $E_P$ by the natural representation of $P$ on the vector space $\mathsf{s}:=\mathsf{g}/ \mathsf{p}$, where $\mathsf{g}$ and $\mathsf{p}$ are Lie algebras of $G$ and $P$, then $\text{deg}E_{\mathsf{s}}<0$. More over, we have similar concept of $S^{\mu}$. For details and proof, please refer to \cite{KN97}.

\begin{prop}\label{vector bundle normal}

Let $\mathcal{E}$ be a flat family of vector bundles or principal $G$ bundles over $S\times C$. Assume that for each closed point $t\in S$, the Kodaira-Spencer maps $$T_tS\rightarrow \text{Ext}^1(\mathcal{E}_t, \mathcal{E}_t)\text{\ \ \ \ or\ \ \ \ }T_tS\rightarrow H^1(C, \mathcal{E}_t(Ad))$$ are surjective. Then:
\begin{enumerate}

\item[(1)] In the vector bundle case, for any $s\in S^{\mu}$, the normal space $N_sS^{\mu}$ is isomorphic to $\text{Ext}^1(F_0,\mathcal{E}_s/F_0)$, where $F_0$ is a maximal stable bundle described above.

\item[(2)] In the principal $G$ bundles case, for any $s\in S^{\mu}$, the normal space $N_sS^{\mu}$ is isomorphic to $H^1(C, \mathcal{E}_{s,\mathsf{s}})$ where $\mathcal{E}_{s,\mathsf{s}}$ is described above.

\end{enumerate}

\end{prop}

\begin{proof}

For the vector bundle case, we first consider the Quot-scheme $\phi:Q\rightarrow S$ parametrizing all subbundles of type $\mu$, then analyse the exact sequence \ref{exact sequence}. Firstly the image of $\phi$ covers $S^{\mu}$, we see that the map $T_qQ\rightarrow T_sS$ factors as $T_qQ\twoheadrightarrow T_sS^{\mu} \hookrightarrow T_sS$. Secondly, by the proof of exactness of \ref{exact sequence}, we see that the map $T_sS\rightarrow \text{Ext}^1(F_0,\mathcal{E}_s/F_0)$ indeed factors as $$T_sS\rightarrow\text{Ext}^1(\mathcal{E}_s, \mathcal{E}_s) \rightarrow \text{Ext}^1(F_0,\mathcal{E}_s/F_0).$$ The first map is Kodaira-Spencer map which is surjective by assumption; the second map is induced by the exact sequence: $$0\longrightarrow F_0\longrightarrow \mathcal{E}_s \longrightarrow \mathcal{E}_s/F_0 \longrightarrow0$$ which is surjective naturally. Thus we see that $\text{Ext}^1(F_0,\mathcal{E}_s/F_0)$ is isomorphic to the cokernel of $T_qQ\rightarrow T_sS$, i.e. the normal space $N_sS^{\mu}$.

The principal bundle case is similar, except we need a variety to parametrize all reductions to $P$. But this is already done in \cite{R96}, it is an open subscheme $\mathcal{U}$ of $\text{Hilb}_{(\mathcal{E}/P)_{/S}}$, parametrizing all sections of $\mathcal{E}/P\rightarrow S$. Now we apply Proposition \ref{begin prop} to this $\mathcal{U}$, with similar method above, we have our proposition.
\end{proof}

\begin{cor}

With same notation and assumptions as above, if we assume $S$ is smooth, we have:

\begin{enumerate}

\item[(1)] In the vector bundle case, the rank of $\mathcal{E}$ is assumed to be $r$, then we have
\begin{align*}
 \text{codim}(S^{us})&\geq (r-1)(g-1)\\
 \text{codim}(S^{uss})&> (r-1)(g-1)
 \end{align*}

\item[(2)] In the principal bundle case, we have
\begin{align*}
 \text{codim}(S^{us})&\geq \text{rank}(\mathcal{E}_{t,\mathsf{s}})(g-1)\\
 \text{codim}(S^{uss})&> \text{rank}(\mathcal{E}_{t,\mathsf{s}})(g-1)
 \end{align*}

\end{enumerate}

\end{cor}

\begin{proof}

Since $S^{\mu}$ is non-empty for only finitely many $\mu$, by proposition above, we only need to calculate $\text{dim}\text{Ext}^1(F_0,\mathcal{E}_t/F_0)$ and $\text{dim}H^1(C, \mathcal{E}_{t,\mathsf{s}})$. Using Riemann-Roch, we have
\begin{align*}
\text{dim}\text{Ext}^1(F_0,\mathcal{E}_t/F_0)&=\text{dim}\text{Hom}(F_0,\mathcal{E}_t/F_0)-\text{deg}\mathcal{H} om(F_0,\mathcal{E}_t/F_0)+r'(r-r')(g-1)\\
\text{dim}H^1(C, \mathcal{E}_{t,\mathsf{s}})&=\text{dim}H^0(C, \mathcal{E}_{t,\mathsf{s}})-\text{deg}\mathcal{E}_{t,\mathsf{s}}+\text{rank}\mathcal{E}_{t,\mathsf{s}}(g-1)
\end{align*}
where $r'$ is the rank of $F$. Thus our corollary holds by analyse of degrees of $\mathcal{H} om(F_0,\mathcal{E}_t/F_0)$ and $\mathcal{E}_{t,\mathsf{s}}$ before.
\end{proof}

\subsection{The case of parabolic vector bundle}

We fix $\mathcal{E}$ to be a flat family of parabolic vector bundles of type $\sigma$ over $S\times C$.To apply our method to parabolic vector bundle case, we need to construct an $S$-scheme parametrizing all flat quotients of $\mathcal{E}$, with fixed parabolic type $\sigma'$.

We begin with a functor $$\mathsf{F}:(Sch/S)^{op}\longrightarrow (Set)$$ as follows: for any $f:T\rightarrow S$, $\mathsf{F}(f:T\rightarrow S)$ is the set of isomorphism classes of all quotients $f^*_C\mathcal{E}\rightarrow \mathcal{G}\rightarrow 0$, such that the induced parabolic structure on $\mathcal{G}$ makes $\mathcal{G}$ a flat family of parabolic vector bundle of rank $r'$ and degree $d'$ with fixed type $\sigma'$.

\begin{prop}

$\mathsf{F}$ is represented by a finite type scheme $\phi_P: Q_P\rightarrow S$.

\end{prop}

\begin{proof}

Thanks to Proposition \ref{equivalence between pa and or}, we will translate parabolic bundle and orbifold bundle interchangeably.

$\mathcal{E}$ gives a flat family of orbifold bundle $\mathcal{W}$ over $S\times Y$. Firstly we consider the Quot-scheme $Q\rightarrow S$, parametrizing all flat quotients of $\mathcal{W}$ with certain fixed rank and degree. Secondly, since $\mathcal{W}$ is an orbifold bundle, we see that $\Gamma$ acts on $Q$, and the closed subscheme $Q^{\Gamma}$ of $\Gamma$-invariant points parametrizes all the orbifold quotients of $\mathcal{W}$(\cite{Se11}). At last, by\cite{Se11} again, there is an open subscheme $Q_P\subset Q^{\Gamma}$, parametrizing all locally free orbifold quotients with fixed type $\sigma'$. We claim that $Q_P$ represents $\mathsf{F}$.

For any $f: T\rightarrow S$, and any quotient $f^*_C\mathcal{E}\rightarrow \mathcal{G}\rightarrow 0$, using the correspondence in  Proposition \ref{equivalence between pa and or}, we see easily that there is an $S$-morphism: $T\rightarrow Q_P$. Conversely, Given an $S$-morphism $\varphi: T\rightarrow Q_P$, this would give a flat orbifold bundle quotient $f_Y^*\mathcal{W}\rightarrow \tilde{\mathcal{G}}\rightarrow 0$. By our correspondence, we have a quotient $$f^*_C\mathcal{E}\rightarrow \mathcal{G}\rightarrow 0$$ where $\mathcal{G}$ is a flat family of parabolic vector bundles with type $\sigma'$. Notice that this is a quotient since taking $\Gamma$ invariant sections of $\mathbb{C}$-modules is an exact functor.
\end{proof}

\begin{rem}

In \cite{G04}, a similar scheme is constructed in a different way.

\end{rem}

\begin{cor}\label{parabolic exact sequence}

For any $s\in S$ and $q\in \phi_P^{-1}(s)$, corresponding to exact sequence: $$0\longrightarrow F\longrightarrow \mathcal{E}_s\longrightarrow G\longrightarrow 0$$ Then we have an exact sequence: $$0\longrightarrow \text{Hom}_{par}(F,G)\longrightarrow T_qQ_P\longrightarrow  T_sS\longrightarrow \text{Ext}^1_{par}(F, G).$$

\end{cor}

\begin{proof}

Let $0\rightarrow \tilde{F} \rightarrow \mathcal{W}_s\rightarrow \tilde{G} \rightarrow 0$ be the corresponding exact sequence of orbifold bundles over $Y$. When we regard $q$ as a point of $Q$, apply the exact sequence \ref{exact sequence}, we have an exact sequence: $$0\longrightarrow \text{Hom}(\tilde{F},\tilde{G})\longrightarrow T_qQ\longrightarrow  T_sS\longrightarrow \text{Ext}^1(\tilde{F},\tilde{G})$$
However, this sequence is in fact a $\Gamma$-exact sequence, Thus we have: $$0\longrightarrow \text{Hom}(\tilde{F},\tilde{G})^{\Gamma}\longrightarrow (T_qQ)^{\Gamma}\longrightarrow  T_sS\longrightarrow \text{Ext}^1(\tilde{F},\tilde{G})^{\Gamma}$$ which is exact since taking $\Gamma$-invariant sections of $\mathbb{C}$-modules is an exact functor. Now, it is known that $\text{Hom}(\tilde{F},\tilde{G})^{\Gamma}=\text{Hom}_{par}(F,G)$ and $(T_qQ)^{\Gamma}=T_qQ_P$. Finally, spectral sequence argument tells $\text{Ext}^1(\tilde{F},\tilde{G})^{\Gamma}=\text{Ext}^1_{par}(F, G)$, we are done.
\end{proof}

Before going further, we mention that there are Harder-Narasimhan filtration and Jordan-Holder filtration for parabolic bundles. So similar as in the previous subsection, for a parabolic bundle which is not stable (resp. semistable), there is a maximal stable subbundle $F_0$ such that $pardeg\mathcal{H}om_{par}(F_0,E/F_0)\leq0$ (resp. $<0$). Moreover, for a family of parabolic vector bundle as above, $S^{\mu}$ defined as before, is locally closed and non-empty for finitely many $\mu$.

\begin{prop}

Assume that for any $t\in S$, the Kodaira-Spencer map $$T_tS\longrightarrow \text{Ext}^1_{par}(\mathcal{E}_t,\mathcal{E}_t)$$ is surjective. Let $S^{\mu}\subset S$ be the locally closed described before. Then for any $s\in S^{\mu}$, we have $N_sS^{\mu}\cong \text{Ext}^1_{par}(F_0,\mathcal{E}_s/F_0)$.

\end{prop}

\begin{proof}

Similar as Proposition \ref{vector bundle normal}.
\end{proof}

\begin{cor}

With same assumption as above, assuming that $S$ is smooth and $\text{rank}\ \mathcal{E}=r$ we have
\begin{align*}
 \text{codim}(S^{us})&\geq \text{deg}D/K+(r-1)(g-1)\\
 \text{codim}(S^{uss})&> \text{deg}D/K+(r-1)(g-1).
 \end{align*}

\end{cor}

\begin{proof}

As before, it suffice to estimate $\text{dim}\text{Ext}^1_{par}(F_0,\mathcal{E}_s/F_0)$. By \cite{Yo95}, we have $\text{Ext}^1_{par}(F_0,\mathcal{E}_s/F_0 )=H^1(C, \mathcal{H}om_{par}(F_0,\mathcal{E}_s/F_0))$, so $$\text{dim}\text{Ext}^1_{par}(F_0,\mathcal{E}_s/F_0)=\text{dim}\text{Hom}_{par}(F_0,\mathcal{E}_s/F_0)-\text{deg}\mathcal{H}om_{par}(F_0,\mathcal{E}_s/F_0)+r'(r-r')(g-1)$$ Since $\text{pardeg}\mathcal{H}om_{par}(F_0,\mathcal{E}_s/F_0)\leq 0$. We see that $-\text{deg}\mathcal{H}om_{par}(F_0,\mathcal{E}_s/F_0)\geq \text{deg}D/K$. This would give our results.
\end{proof}

\begin{rem}

Similar results have been given in \cite{Sun00} by a different way.

\end{rem}

\subsection{The case of parabolic symplectic/orthogonal bundle}
The case of parabolic symplectic/orthogonal bundles is similar to those in former two sections, but we need define some notions first.

Let $E$ be a parabolic symplectic bundle over $C$, and $W$ be the corresponding orbifold symplectic bundle over $Y$. By the constructions before, we have $W(Ad)$ and $W_{\mathsf{s}}$ for $\mathsf{s}=\mathsf{g}/\mathsf{p}$. $W$ is an orbifold symplectic bundle, so $W(Ad)$ and $W_{\mathsf{s}}$ are both orbifold vector bundles over $Y$. We use $E(Ad)$ and $E_{\mathsf{s}}$ to denote corresponding parabolic vector bundles over $C$.

For any family of parabolic symplectic bundle $\mathcal{E}$ over $C$ parametrized by a scheme $S$, let $\mathcal{W}$ be the corresponding orbifold symplectic bundle on $S\times Y$. For any $t\in S$, we have the Kodaira-Spencer map $$T_tS\longrightarrow \text{H}^1(Y, \mathcal{W}_t(Ad))$$ for $\mathcal{W}$. This map is obviously $\Gamma$-invariant, so we have $$T_tS\longrightarrow \text{H}^1(Y, \mathcal{W}_t(Ad))^{\Gamma}=\text{H}^1(C, \mathcal{E}_t(Ad))$$

\begin{defn}

The Kodaira-Spencer map for $\mathcal{E}$ at $t\in S$ is given by $$T_tS\longrightarrow \text{H}^1(C, \mathcal{E}_t(Ad)).$$

\end{defn}

\begin{prop}

Let $S$ and $\mathcal{E}$ be as before. Then there is a scheme $\phi_{PS}: Q_{PS}\rightarrow S$ parametrizing all isotropic subbundles of $\mathcal{E}$, flat over $S$ with same fixed type $\tau'$.

Moreover, for any $s\in S$ and $q\in \phi_{PS}^{-1}(s)$, corresponding to an isotropic subbundle $F\subset \mathcal{E}_s$, which corresponds to a reduction to a parabolic subgroup $P$ of $\mathcal{W}_s$, we have an exact sequence: $$0\longrightarrow \text{H}^0(C,\mathcal{E}_{s,\mathsf{s}})\longrightarrow T_qQ_{PS}\longrightarrow T_sS\longrightarrow \text{H}^1(C,\mathcal{E}_{s,\mathsf{s}})$$

\end{prop}

\begin{proof}

Similar to Corollary \ref{parabolic exact sequence}.
\end{proof}

With similar method, we can show that:

\begin{cor}\label{codimension G}

With notations as before, assume that the Kodaira-Spencer map is surjective for any $s\in S$, then we have
\begin{align*}
 \text{codim}(S^{us})&\geq \text{deg}D/K+\text{rank}(\mathcal{E}_{s,\mathsf{s}})(g-1)\\
 \text{codim}(S^{uss})&> \text{deg}D/K+\text{rank}(\mathcal{E}_{s,\mathsf{s}})(g-1).
 \end{align*}

\end{cor}

\section{Infinite Grassmannians and the theta line bundle}\label{Inf Gr}

\subsection{Infinite Grassmannians}

In this subsection, we use $G$ to denote a connected simply connected simple affine algebraic group, and the parabolic $G$ bundle over $C$ we considered in this subsection is given by a principal $G$ bundle $E$ together with choices of one parameter subgroups in $E(G)_x$ for every $x\in D$; a quasi-parabolic $G$ bundle is just a choice of choices of parabolic subgroups of $E(G)_x$, i.e. (quasi-)parabolic $G$ bundles in the sense of \cite{BR89}.

We fix a point $p\in X$, away from $D$, let $C^*=C-p$, following \cite{KNR94}, we define
\begin{align*}
\mathcal{G}&=G(\hat{\mathbb{C}}_p)\\
\mathcal{P}&=G(\hat{\mathcal{O}}_p)\\
\Lambda &=G(\mathbb{C}[C^*])
\end{align*}
where $\hat{\mathcal{O}}_p$ is the completion of local ring $\mathcal{O}_p$ of $p\in C$; $\hat{\mathbb{C}}_p$ is the field of quotient of $\hat{\mathcal{O}}_p$; $\mathbb{C}[C^*]$ is the coordinate ring of $C^*$. Similarly in \cite{KNR94}, we have

\begin{prop}\label{bijection}
If we use $\mathcal{X}$ to denote the set of isomorphism classes of quasi-parabolic $G$ bundle with parabolic structure $P_x$ at each $x\in D$, we have a bijection of sets:
$$\alpha: \Lambda\setminus (\mathcal{G}/\mathcal{P}\times \prod_{x\in D}G/P_{x})\longrightarrow \mathcal{X}$$

\end{prop}

\begin{proof}

By proposition 1.5 of \cite{KNR94}, there is a bijection between $\Lambda\setminus \mathcal{G}/\mathcal{P}$ and the set of isomorphism classes of $G$ bundles. Notice that a quasi-parabolic $G$ bundle is nothing but a parabolic $G$ bundle plus a point in $\prod_{x\in D}G/P_{x}$, we have our bijection.
\end{proof}

Recall that in \cite{KNR94} the generalized flag variety $X:=\mathcal{G}/\mathcal{P}$ has a structure of ind-variety, more precisely, $$X=\lim_{\rightarrow}X_{\sigma}$$
where $X_{\sigma}$ are the generalised Schubert varieties they defined there. Moreover, there is an algebraic $G$ bundle $\mathcal{U}\rightarrow C\times X$ such that $\mathcal{U}|_{C^{*}\times X}$ is trivial. So, for any $x\in D$, we have a trivial $G$ bundle $\mathcal{U}_x$ over $X$, then we define $$X_P=X\times \prod_{x\in D}G/P_{x}$$ to be the relative flag variety over $X$ defined by $\{\mathcal{U}_x\}_{x\in D}$. Let $\pi:X_P\longrightarrow X$ be the natural projection.

\begin{prop}

There is a quasi-parabolic $G$ bundle $\mathcal{U}_P$ over $C\times X_P$ such that for any $x\in X_P$, the quasi-parabolic bundle $(\mathcal{U}_P)_x:=\mathcal{U}_P|_{C\times x}$ is exactly the parabolic $G$ bundle corresponds to $x$ though the bijection in Proposition \ref{bijection}. Moreover the bundle $\mathcal{U}|_{C^*\times X_P}$ carries a trivialization $\epsilon: \tau\rightarrow \mathcal{U}|_{C^*\times X_P} $ where $\tau$ is a trivial quasi-parabolic $G$ bundle over $C^*\times X_P$.

For any scheme $T$ and any family of parabolic $G$ bundle $\mathcal{F}$ over $C\times T$, if $\mathcal{F}|_{C^*\times T}$ and $\mathcal{F}|_{Spec\hat{\mathcal{O}}_p\times T}$ are both trivial. Then if we choose a trivialization $\varepsilon: \tau'\rightarrow \mathcal{F}|_{C^*\times T}$, we would have a Schubert variety $X_{\sigma}$, and a morphism $f:T\rightarrow X_{\sigma}\times \prod_{x\in D}G/P_{x}$ such that $\varepsilon$ is exactly the trivialization pulled back from $\epsilon$ by $f$.

\end{prop}

\begin{proof}

This is just a parabolic analogy of proposition 2.8 in \cite{KNR94}.
The quasi-parabolic $G$ bundle $\mathcal{U}_P$ is given by $\pi^*\mathcal{U}$ with quasi-parabolic structure determined by universal property of flag variety.

To see the existence of the morphism $f$, we firstly observe that by proposition 2.8 in \cite{KNR94}, we have a morphism $f^{\prime}:T\longrightarrow X$. Now since $\mathcal{F}|_{C^*\times T}$is trivial, we would have a point in $\prod_{x\in D}G/P_{x}$ determined by this trivial parabolic $G$ bundle.
\end{proof}

\begin{cor}

There is an open subset $X_P^{ss}\subset X_P$ and a morphism $\phi: X_P^{ss}\longrightarrow M_{G,P}$ to the moduli space of semistable $G$ bundles.

\end{cor}

By proposition \ref{bijection}, for any point $m\in M_{G,P}$, the fibre $\phi^{-1}(m)$ is a union of certain $\Lambda$-orbits. Next, we analyse the closure of these orbits.

\begin{lem}\label{S equivalence}

Let $E$ be a semistable parabolic $G$ bundle on $C$ and we consider $gr(E)$ defined in Proposition 3.1 of \cite{BR89}. Then there exists a family of parabolic $G$ bundle $\mathcal{E}$ on $C\times \mathbb{A}^1$ such that:
\begin{enumerate}
\item[(a)] $\mathcal{E}|_{C\times (\mathbb{A}^1\setminus \{0\})}\cong p_C^*(E)$ , $\mathcal{E}|_{C\times \{0\}}\cong gr(E)$ and

\item[(b)] $\mathcal{E}|_{C^*\times \mathbb{A}^1}$ and $\mathcal{E}|_{Spec \hat{\mathcal{O}}_p \times \mathbb{A}^1}$ are both trivial.
\end{enumerate}
Where $p_C$ is the projection from $C\times \mathbb{A}^1$ to $C$.

\end{lem}

\begin{proof}

This is the parabolic analogy of Proposition 3.7 of \cite{KNR94}. Proof is similar and we omit the it here.
\end{proof}

Now we have the following:

\begin{prop}\label{injection}

The morphism $\phi^{*}: \text{Pic}(M_{G,P})\longrightarrow \text{Pic}(X_P^{ss})$ is injective.

\end{prop}

\begin{proof}

By lemma \ref{S equivalence} we know that the fibre $\phi^{-1}(m)$ for any $m\in M_{G,P}$ is a disjoint of $\Lambda$ orbits, and the closure of these orbits intersect with each other. Thus by a similar argument in the proof of lemma 2.1 in \cite{KN97}, we have our injection.
\end{proof}

\subsection{The theta line bundle and the canonical line bundle of $M_{G,P}$}

In this subsection, we fix $$l:=\dfrac{1}{r}(K \chi -\sum_{x\in D}\sum_{i=1}^{2l_x}d_i(x)r_i(x))$$ to be an integer, and $D_l:=\sum_ql_qz_q$ to be an effective divisor of degree $l$ on $C$.

Given a scheme $S$ and a flat family of parabolic principal $G$ bundle $\mathcal{F}$ over $S\times C$ with parabolic type $\big(\{\overrightarrow{n}(x)\}_{x\in D}, \{\overrightarrow{a}(x)\}_{x\in D}\big)$, assuming that for each $x\in D$, the filtration is given by $$0=F_{2l_x+1}(\mathcal{F}_{S\times \{x\}})\subseteq \cdots F_{l_x+1}(\mathcal{F}_{S\times \{x\}})\subseteq F_{l_x}(\mathcal{F}_{S\times \{x\}})\subseteq \cdots \subseteq F_0(\mathcal{F}_{S\times \{x\}})=\mathcal{F}_{S\times \{x\}}$$ which is equivalent to $$\mathcal{F}_{S\times \{x\}}=Q_{2l_x+1}(\mathcal{F}_{S\times \{x\}})\twoheadrightarrow \cdots \twoheadrightarrow Q_{l_x+1}(\mathcal{F}_{S\times \{x\}})\twoheadrightarrow Q_{l_x}(\mathcal{F}_{S\times \{x\}})\twoheadrightarrow \cdots \twoheadrightarrow Q_0(\mathcal{F}_{S\times \{x\}})=0$$ then we can define a line bundle $\Theta_{\mathcal{F},D_l}$ on $S$ by $$\Theta_{\mathcal{F},D_l}:=(\text{det}R\pi_S\mathcal{F})^{-K}\otimes \bigotimes_{x\in D}\big\{\bigotimes_{i=1}^{2l_x}\text{det}\big(Q_{i}(\mathcal{F}_{S\times \{x\}})\big)^{d_{i}(x)}\big\}\otimes \bigotimes_{q}\text{det}(\mathcal{F}_{S\times \{z_q\}})^{l_q}$$ where $\pi_S: S\times C\rightarrow S$ is the projection and $\text{det}R_{\pi_S}\mathcal{F}$ is the determinant of cohomology: $\{\text{det}R_{\pi_S}\mathcal{F}\}_t=\text{det}\text{H}^0(C,\mathcal{F}_t)\otimes \text{det}\text{H}^1(C,\mathcal{F}_t)^{-1}$. Notice that $$\text{det}\big(Q_{i}(\mathcal{F}_{S\times \{x\}})\big)\cong \text{det}\big(Q_{2l_x+1-i}(\mathcal{F}_{S\times \{x\}})\big)$$ for $1\leq i \leq l_x$.

It is clear that for any morphism $f:T\rightarrow S$, we have $f^*\Theta_{\mathcal{F},D_l}=\Theta_{f_C^*\mathcal{F},D_l}$, where $f_C:T\times C\rightarrow S\times C$ is the base change of $f$. Moreover, we have:

\begin{thm}\label{theta line bundle}

There is a unique ample line bundle $\Theta_{D_l}$ over the moduli space $M_{G,P}$, such that:

\begin{enumerate}

\item[(1)] For any scheme $S$ and any family of semistable parabolic $G$ bundle $\mathcal{F}$ over $S\times C$, let $\phi_{\mathcal{F}}:S\rightarrow M_{G,P}$ be the induced map, then we have $$\phi_{\mathcal{F}}^*\Theta_{D_l}=\Theta_{\mathcal{F},D_l}.$$

\item[(2)] Let $D_l$ and $D_l^{\prime}$ be two different effective divisor of degree $l$ on $C$, then $\Theta_{D_l}$ and $\Theta_{D_l^{\prime}}$ are algebraically equivalent.

\end{enumerate}

\end{thm}

\begin{proof}

$\Theta_{D_l}$ is the descent of $\Theta_{\mathcal{E},D_l}$ over $\mathcal{R}^{ss}$ for the universal parabolic symplectic/orthogonal bundle. The reason of descent of $\Theta_{\mathcal{E},D_l}$ is the same as the parabolic bundle case as in \cite{SZh18}, \cite{P96} once we see that the pull back of polarization over $\textbf{P}Hom(V\otimes V, H)$ to $\mathcal{R}^{ss}$ is trivial. Similarly we can show $\Theta_{D_l}$ is ample and for different choice of $D_l$, the theta line bundles are algebraically equivalent.
\end{proof}

For any parabolic $G$ bundle $E$, with parabolic structure $t_x\in G/P_x, \forall x \in D$, we define $\mathbf{D}_{E}$ to be  the space of infinitesimal deformation of $E$, i.e. the space of isomorphism classes of parabolic $G$ bundles $\tilde{E}$ on $C[\epsilon]$, such that $\tilde{E}|_{C}\cong E$, where $C[\epsilon]=C\times Spec(\mathbb{C}[\epsilon]/(\epsilon^2)) $.

\begin{prop}

There is an exact sequence: $$0\longrightarrow \prod_{x\in D}T_{t_x}(G/P_x)\stackrel{f}{\longrightarrow} \mathbf{D}_{E}\stackrel{g}{\longrightarrow} H^1(C, E(Ad))\longrightarrow 0$$ where $T_{t_x}(G/P_x)$ is the tangent space of $G/P_x$ at $t_x$.

\end{prop}

\begin{proof}

Recall that $H^1(C, E(Ad))$ is the infinitesimal deformation space of $E$ as a twisted $G$ bundle, so the morphism $g$ is given by forgetting parabolic structures. Since every twisted $G$ bundle can be equipped with any parabolic structure, $g$ is an surjection.

To determine the kernel of $g$, we need to figure out how many parabolic structures we can impose on a $\mathcal{E}$ so that the restriction to $C$ are the parabolic structures $\{t_x\in G/P_x\}$. The question is local, so it is equivalent to find a parabolic subgroups $\tilde{P}_x\subset G(\mathbb{C}[\epsilon]/(\epsilon^2))$ such that $\tilde{P}_x|_0=t_x\in G/P_x$. The space of such groups is exactly $\prod_{x\in D}T_{t_x}(G/P_x)$.
\end{proof}

\begin{cor}

For any family of stable parabolic $G$ bundle $\mathcal{F}$ over $S\times  C$,let $\pi_S: S\times C\longrightarrow S$ be the projection and $\varphi_S:S\longrightarrow M_{G,P}$ be the induced map, then

$$\varphi_S^*(\omega_{M_{G,P}}^{-1})=\text{det}(R\pi_S\mathcal{F}(Ad))^{-1}\otimes \bigotimes_{x\in D}\big\{\bigotimes_{i=1}^{l_x}\text{det}\big(Q_{i}(\mathcal{F}_{S\times \{x\}})\big)^{m_i(x)} \big\}$$ where $m_i(x)=n_i(x)+n_{i+1}(x)$ for $1\leq x \leq l_x-1$;  $m_{l_x}=n_{l_x}+n_{l_x+1}+1$ for $G=Sp(2n)$; $m_{l_x}=n_{l_x}+n_{l_x+1}-1$ for $G=SO(2n)$ and $m_{l_x}=n_{l_x}+n_{l_x+1}$ for $G=SO(2n+1)$.

\end{cor}

The main results in this section is to under certain choices of weights, the moduli space of parabolic symplectic/orthogonal bundles are Fano varieties. A normal projective variety $X$ is call Fano if $\omega_X^{-1}$ is an ample line bundle. Our method is to compare the pull back of anti-canonical line bundle over $M_{G,P}$ to $X_{P}^{ss}$ with theta line bundle over $X_{P}^{ss}$. It is known that the Picard group of moduli space of symplectic/orthogonal bundles has rank one, so there exists positive integer $\chi_G$ such that $\text{det}(R\pi_S\mathcal{F}(Ad))\cong(\text{det}R\pi_S\mathcal{F})^{\otimes\chi_G}$. For $G=Sp(2n)$, $\chi_G=n+1$, for $G=SO(2n)$, $\chi_G=2n-2$ and for $G=SO(2n+1)$, $\chi_G=2n-1$.

We first deal with symplectic case, since symplectic groups are simply connected. Combine Proposition \ref{injection} and Theorem  \ref{theta line bundle} together, we have:

\begin{prop}\label{Fano1}

Let $G=Sp(2n)$ , $K=2\chi_G$ and $\overrightarrow{a}(x)$ satisfying $a_{i+1}(x)-a_{i}(x)=m_i(x)$ for $1\leq i \leq l_x$ , the moduli space of parabolic symplectic bundles are Fano.

\end{prop}

\begin{proof}

We show that under the condition in the proposition, $\Theta_{D_l}$ is equal to $\omega_{M_{G,P}}^{-2\chi_G}$. The problem here is that we do not know whether $M_{G,P}$ is Gorenstein or not, i.e. whether $\omega_{M_{G,P}}$ is a line bundle. But we do know that $M_{G,P}$ is Cohen-Macaulay and normal. Let $M^{\circ}\subset M_{G,P}$ be the open subset where $\omega_{M_{G,P}}$ is a line bundle and points in $M^{\circ}$ representing stable bundles, then we have $\text{codim}(M_{G,P}\setminus M^{\circ})\geq 2$. Apply Proposition \ref{injection} to $M^{\circ}$ we see that $\omega_{M_{G,P}}^{-2\chi_G}$ and $\Theta_{D_l}$ are coincide over $M^{\circ}$. Now we use Lemma 2.7 of \cite{KN97}, and we see that $\omega_{M_{G,P}}$ is a line bundle, moreover, $M_{G,P}$ is a Fano variety.
\end{proof}

The special orthogonal group case is different, since $SO(n)$ is not simply connected, and its universal cover is $Spin(n)$. For any one parameter subgroup of $SO(n)$, we choose a lift to be a one parameter subgroup of $Spin(n)$. Then if we consider the moduli space of parabolic $Spin(n)$ bundles with parabolic structure given by the lifts, by Lemma 1.4 of \cite{BR89}, we would have a natural map: $t: M_{Spin(n),P}\rightarrow M_{SO(n),P}$ which identifies $M_{SO(n),P}$ as a quotient by a finite group of $M_{Spin(n),P}$. By discussion in the section 6 of \cite{BLS98}, we have:

\begin{prop}

The map between Picard groups: $t^*:\text{Pic}(M_{SO(n),P})\rightarrow \text{Pic}(M_{Spin(n),P})$ is injective on the subgroup of infinite order elements.

\end{prop}

Similar as before, we have:

\begin{prop}\label{Fano2}

Let $G=SO(n)$ , $K=2\chi_G$ and $\overrightarrow{a}(x)$ satisfying $a_{i+1}(x)-a_{i}(x)=m_i(x)$ for $1\leq i \leq l_x$ , the moduli space of parabolic special orthogonal bundles are Fano.

\end{prop}

\section{Globally F regular type varieties and Main theorem}\label{section 6}

Let $k$ be a perfect field of $char(k)=p>0$ and $X$ be a normal variety over $k$. Consider $$F: X\longrightarrow X$$ to be the absolute Frobenius map and $F^e: X\rightarrow X$ to be the $e$-th iteration of $F$.

For any Weil divisor $D\in Div(X)$, we have a reflexive sheaf $$\mathcal{O}_X(D)=j_*\mathcal{O}_{X^{sm}}(D)$$ where $j:X^{sm}\hookrightarrow X$ is the inclusion of smooth locus, and $\mathcal{O}_X(D)$ is an invertible sheaf if and only if $D$ is a Cartier divisor.

\begin{defn}

Let $X$ and $D$ be as above, $X$ is called \emph{stably Frobenius D-split} if the natural homomorphism $$\mathcal{O}_X\longrightarrow F^e_*\mathcal{O}_X(D)$$ is split as an $\mathcal{O}_X$ homomorphism for some $e>0$. And $X$ is called \emph{globally F-regular} if $X$ is stably Frobenius $D$-split for any effective divisor $D$.

\end{defn}

We state the following lemma about globally F-regular varieties, for proof and more details, please refer to \cite{SZh18}, \cite{S00}.

\begin{lem}[Corollary 6.4 of \cite{SS10}]\label{generating lemma}

Let $f:X\rightarrow Y$ be a morphism of normal varieties over $k$. Assume that the natural map $$f^{\#}: \mathcal{O}_Y\longrightarrow f_*\mathcal{O}_X$$ splits as an $\mathcal{O}_Y$ homomorphism, then if $X$ is globally F-regular, so is $Y$.

\end{lem}

Now we let $K$ be a field of characteristic zero.

For any scheme $X$ over $K$, there is a finitely generated $\mathbb{Z}$-algebra $R\subset K$ such that $X$ is "defined" over $R$. That is, there is a flat $R$-scheme $$X_R\longrightarrow S=\text{Spec}R$$ such that $X_K:=X_R\times_S \text{Spec}K\cong X$. $X_R\longrightarrow S$ is called an integral model of $X/K$. For any closed point $s\in S$, $X_s:=X_R\times_S \text{Spec}(\overline{k(s)})$ is called the "modulo $p$ reduction" of $X$, where $p=char(k(s))>0$.

\begin{defn}

A variety $X$ over $K$ is called of \emph{globally F-regular type} if its "modulo $p$ reduction" of $X$ are globally F-regular for a dense set of $p$ for some integral model $X_R\rightarrow S$.

\end{defn}

Globally F-regular type varieties have many nice properties, which we will state some of them as the following theorem. Again, for proof and more details, please refer to \cite{SZh18} and \cite{S00}.

\begin{thm}

Let $X$ be a projective variety over $K$, if $X$ is of globally F-regular type, then:

\begin{enumerate}
\item[(1)] $X$ is normal, Cohen-Macaulay with rational singularities. If $X$ is $\mathbb{Q}$-Gorenstein, then $X$ has log terminal singularities.

\item[(2)] For any nef line bundle $\mathcal{L}$ over $X$, we have ${\rm H}^i(X, \mathcal{L})=0$, for any $i>0$. In particular, ${\rm H}^i(X, \mathcal{O}_X)=0$ for any $i>0$.
\end{enumerate}

\end{thm}

Our main theorem of this paper is:

\begin{thm}\label{Main thm}

The moduli space of parabolic symplectic/orthogonal bundles $M_{P}$ over a smooth projective curve $C$ over $\mathbb{C}$ is of globally F-regular type.

\end{thm}

\begin{cor}

Let $\Theta_{D_l}$ be the theta line bundle over $M_{G,P}$ define before, then $${\rm H}^i(M_P, \Theta_{D_l})=0$$ for any $i>0$.

\end{cor}

Our beginning example of globally F-regular type variety is Fano variety.

\begin{prop}[Proposition 6.3 in \cite{S00}]

A Fano variety over $K$ with at most rational singularities is of globally F-regular type.

\end{prop}

With our beginning example, the next step is to ask whether Lemma \ref{generating lemma} holds in characteristic zero. To answer such question, in \cite{SZh18}, they introduced the following:

\begin{defn}

A morphism $f: X\rightarrow Y$ of varieties over $K$ is called \emph{$p$-compatible} if there is an integral model $$f_R: X_R\longrightarrow Y_R$$ such that, if for any $s\in S=\text{Spec}R$, we put $X_s=X_R\times_S\text{Spec}\overline{k(s)}$ , $Y_s=Y_R\times_S\text{Spec}\overline{k(s)}$ and consider
\begin{equation*}
\xymatrix{
X_s\ar[r]^{j_s} \ar[d]_{f_s}&X_R\ar[d]^{f_R}\\
Y_s\ar[r]^{i_s}&Y_R}
\end{equation*}
then we have that $i_s^*f_{R*}\mathcal{O}_{X_R}=f_{s*}j_s^*\mathcal{O}_{X_R}$ holds for a dense set of $s$.

\end{defn}

It can be shown that if $f: X\rightarrow Y$ is a flat proper morphisms such that $\textbf{R}^if_{*}\mathcal{O}_X=0$ for all $i\geq 1$, then $f$ is $p$-compatible.

To prove our main theorem, we need to introduce a key proposition from \cite{SZh18}.

Let $(\mathcal{R}^{\prime}, L^{\prime})$ and $(\mathcal{R},L)$ be two polarized projective varieties over $K$, with linear actions by a reductive group scheme $G$ over $K$ respectively. We use $(\mathcal{R}^{\prime})^{ss}(L^{\prime})\subseteq \mathcal{R}^{\prime}$ and $\mathcal{R}^{ss}(L)\subseteq \mathcal{R}$ to denote the GIT semistable locus, then there are projective GIT quotients: $$\psi:\mathcal{R}^{ss}(L)\rightarrow Y:=\mathcal{R}^{ss}(L)//G\ ,\ \varphi:(\mathcal{R}^{\prime})^{ss}(L^{\prime})\rightarrow Z:=(\mathcal{R}^{\prime})^{ss}(L^{\prime})//G$$

\begin{prop}[Proposition 2.10 of \cite{SZh18}]\label{key prop}

Let $\mathcal{R}$, $\mathcal{R}^{\prime}$ as above. Considering the following diagram, assume
\begin{enumerate}
\item[(1)] there is a $G$-invariant $p$-compatible morphism $\hat{f}:\mathcal{R}^{\prime}\rightarrow \mathcal{R}$ such that $\hat{f}_*\mathcal{O}_{\mathcal{R^{\prime}}}=\mathcal{O}_{\mathcal{R}}$;
\item[(2)] there is a $G$-invariant open subset $W\subset (\mathcal{R}^{\prime})^{ss}(L^{\prime})$ such that $$\emph{Codim}(\mathcal{R}^{\prime}\setminus W)\geq 2,\ \hat{X}=\varphi^{-1}\varphi(\hat{X})$$ where $\hat{X}=W\cap \hat{f}^{-1}(\mathcal{R}^{ss}(L))$. And we put $X=\varphi(\hat{X})$.
\end{enumerate}
Then if $Z$ is of globally F-regular type, so is $Y$.

\begin{equation*}
\xymatrix{
(\mathcal{R}^{\prime})^{ss}(L^{\prime}) \ar@{^(->}[r] \ar[dd]^{\varphi} & \mathcal{R}^{\prime} \ar[rr]^{\hat{f}} && \mathcal{R} & \mathcal{R}^{ss}(L)\ar[dd]^{\psi} \ar@{_(->}[l]\\
& W\ar@{^(->}[u] \ar@{^(->}[ul] & \hat{X}=W\cap \hat{f}^{-1}(\mathcal{R}^{ss}(L)) \ar@{_(->}[l] \ar[d] \\
Z && X\ar@{_(->}[ll]\ar[rr]^f  && Y}
\end{equation*}

\end{prop}

Finally, we will prove our main theorem:

\begin{proof}[Proof of Theorem \ref{Main thm}]

We choose an effective divisor $D^{\prime}$ of $C$ such that $D^{\prime}\cap D=\emptyset$, deg$D^{\prime}$ being even and $$\dfrac{\text{deg}(D)+\text{deg}(D^{\prime})}{2\chi_G}+(r-1)(g-1)\geq 2$$ and for each $x\in D^{\prime}$, we put $\overrightarrow{n}(x)=(1,\cdots, 1)$. Let $Z^{\prime}$ be the scheme parametrizing symplectic/orthogonal bundles $(E, \omega)$ where $\omega : E\otimes E \rightarrow \mathcal{O}_C(-D-D^{\prime})$ as we constructed in section \ref{construction}. We see that $Z^{\prime}\cong Z$. Then we let $$\mathcal{R}^{\prime}=\underset{x\in D\cup D^{\prime}}{\times_Z}Flag_{\vec n(x)}(\mathcal F_x)=\mathcal{R}\times_Z\big(\underset{x\in D^{\prime}}{\times_Z}Flag_{\vec n(x)}(\mathcal F_x)\big)\stackrel{\hat{f}}{\longrightarrow} \mathcal{R}.$$

So $\hat{f}:\mathcal{R}^{\prime}\rightarrow \mathcal{R}$ is a flag bundle and hence $p$-compatible with $\hat{f}_*\mathcal{O}_{\mathcal{R}^{\prime}}=\mathcal{O}_{\mathcal{R}}$. We choose polarization for $\mathcal{R}^{\prime}$ and $\mathcal{R}$ as the ones given in Section \ref{construction}, say $L^{\prime}$ and $L$ Clearly there are $SL(V)$ action on $\mathcal{R}^{\prime}$ and $\mathcal{R}$ and $\hat{f}$ is $SL(V)$-invariant.

Now we put $K=2\chi_G$ and give weights for $\mathcal{R}^{\prime}$ by $\overrightarrow{a}(x)$ satisfying $a_{i+1}(x)-a_i(x)=m_i(x)$ and $a_{i}(x)+a_{2l_x+2-i}(x)=K$ for $1\leq i \leq l_x$ and any  $x\in D\cup D^{\prime}$. So by Proposition \ref{Fano1} and \ref{Fano2} we see that $Z:=(\mathcal{R}^{\prime})^{ss}(L^{\prime})//SL(V)$ is a Fano variety. We use $\varphi: (\mathcal{R}^{\prime})^{ss}\rightarrow Z$ to denote the quotient map.

Moreover, if one let $W=(\mathcal{R}^{\prime})^{s}$, $\hat{X}=W\cap \hat{f}^{-1}(\mathcal{R}^{ss})$ and $X=\varphi(\hat{X})$ then clearly $\hat{X}=\varphi^{-1}(X)$. By Corollary \ref{codimension G} and our assumption, we would have: $\text{Codim}(\mathcal{R}^{\prime}\setminus W)\geq 2$. Now Proposition \ref{key prop} shows that the moduli space of parabolic symplectic/orthogonal bundles $Y:=\mathcal{R}(L)^{ss}//SL(V)$ is of globally $F$-regular type.
\end{proof}

\bibliographystyle{plain}

\bibliography{ref}

\end{document}